\UseRawInputEncoding %arxiv
\documentclass[11pt,a4paper]{article}
\usepackage{amsfonts}
\usepackage{amssymb}
\usepackage{mathrsfs}
\usepackage{amsmath}
\usepackage{booktabs}
\usepackage{epsf,epsfig,amsfonts,amsgen,indentfirst}
\usepackage{amsmath,amstext,amsbsy,amsopn,amsthm,bbding,wasysym}
\usepackage{multicol,mathdots}
\usepackage{subfigure}
\usepackage[numbers,sort&compress]{natbib}
\allowdisplaybreaks

\newtheorem{theorem}{Theorem}[section]
\newtheorem{corollary}{Corollary}[section]

\newtheorem{lemma}[theorem]{Lemma}
\newtheorem{cor}[theorem]{Corollary}

\newtheorem{definition}{Definition}[section]
\newtheorem{remark}{Remark}[section]

\newtheorem{conjecture}[theorem]{Conjecture}

\begin{document}
\title
{\LARGE \textbf{Orlik--Solomon algebras determine the Tutte polynomials of matroids of rank at most three
\thanks{Supported by NSFC (Nos. 12261071,  12571019) and NSF of Qinghai Province (No. 2025-ZJ-902T).} }}

\author{Jianxuan Luo$^a$, Tingzeng Wu$^{b,c}$\thanks{{Corresponding author.\newline
\emph{E-mail addresses}: mathtzwu@163.com, jianxuanluo@163.com, hjlai2015@hotmail.com.
}}, Hong-Jian Lai$^{d}$ \\
{\small $^{a}$ School of Mathematical Sciences, Xiamen University,}\\
{\small  Xiamen, Fujian 361005,  P.R.~China} \\
{\small $^{b}$ School of Mathematics and Statistics, Qinghai Minzu University, }\\
{\small  Xining, Qinghai 810007, P.R.~China} \\
{\small $^{c}$ Qinghai Institute of Applied Mathematics, Xining, Qinghai 810007, P.R.~China}\\
{\small $^{d}$ School of Mathematics and System Sciences, Guangdong Polytechnic Normal University,}\\
{\small  Guangzhou, Guangdong 510665, P.R.~China}}
\date{}

\maketitle
\noindent {\bf Abstract:}
Let $\mathcal{M}$ be a simple matroid of rank $r$ on the ground set $[n]=\{1,\ldots,n\}$, 
let $T_{\mathcal{M}}(x,y)$ denote its Tutte polynomial, 
and let $\mathrm{OS}(\mathcal{M})$ denote its Orlik--Solomon $\mathbb{C}$-algebra. 
We say that $\mathrm{OS}(\mathcal{M})$ determines $T_{\mathcal{M}}(x,y)$ if, 
for every simple matroid $\mathcal{N}$, 
an isomorphism $\mathrm{OS}(\mathcal{M})\cong\mathrm{OS}(\mathcal{N})$ of $\mathbb{C}$-algebras, 
without regard to the grading, implies $T_{\mathcal{M}}(x,y)=T_{\mathcal{N}}(x,y)$. 
Falk conjectured that the Orlik--Solomon $\mathbb{C}$-algebra 
of every simple matroid of rank three determines its Tutte polynomial. 
We prove that this holds for every simple matroid of rank at most three, 
thereby resolving Falk's conjecture.

%\smallskip
\noindent {\bf Keywords:} matroid; Orlik--Solomon algebra; Tutte polynomial \\
\noindent {\bf AMS subject classifications:} 05B35; 52C35; 16W50

\section{Introduction}

We adopt the notation of Cordovil and Forge~\cite{CordovilForge2003}.
Every matroid appearing in this paper is finite. For each integer
$n\geq 0$, set $[n]=\{1,\ldots,n\}$, with $[0]=\varnothing$.
A matroid is termed \emph{simple} when it contains neither loops nor
parallel elements. Except where explicitly stated otherwise, $\mathcal{M}$
stands for a simple matroid of rank $r$ whose ground set is $[n]$.
Let $\operatorname{rk}_{\mathcal{M}}$ be its rank function.
For $A\subseteq[n]$, the rank $\operatorname{rk}_{\mathcal{M}}(A)$
is defined as the largest cardinality of an independent subset of $A$,
and $r=\operatorname{rk}_{\mathcal{M}}([n])$.
Denote by $\mathfrak{C}(\mathcal{M})$ and $\mathfrak{B}(\mathcal{M})$ the sets
of circuits (that is, minimal dependent sets under inclusion) and bases of $\mathcal{M}$, respectively.
A matroid on $[n]$ is \emph{realizable over a field $\mathbb{K}$} if there
are vectors $v_1,\ldots,v_n$ in a vector space over $\mathbb{K}$ such
that $A\subseteq[n]$ is independent in the matroid exactly when
the vectors $v_i$, $i\in A$, are linearly independent over $\mathbb{K}$.
For further matroid terminology, see~\cite{Oxley2011}.

The Tutte polynomial of $\mathcal{M}$ is defined by the following
subset expansion (see~\cite{BrylawskiOxley1992}):
\begin{equation}\label{eq:tutte-subset-expansion}
T_{\mathcal{M}}(x,y)=\sum_{X\subseteq[n]}(x-1)^{r-\operatorname{rk}_{\mathcal{M}}(X)}(y-1)^{|X|-\operatorname{rk}_{\mathcal{M}}(X)}.
\end{equation}
For further work on Tutte polynomials of graphs and matroids, 
see \cite{BrylawskiOxley1992}.

Let $E=\{e_1,e_2,\ldots,e_n\}$ be a fixed basis of the vector space $\mathbb{C}^n$ and define $\mathcal{E}_0=\mathbb C$, $\mathcal{E}_{\ell}
=\sum \alpha_i e_{i_1}\cdots e_{i_\ell} (\alpha_i\in\mathbb{C})$.
Consider the \emph{graded exterior algebra} of the
vector space $\bigoplus_{i=1}^{n}\mathbb{C}e_i$:
\begin{equation*}
\begin{aligned}
\mathcal{E}
&= \bigwedge\left(\bigoplus_{i=1}^{n}\mathbb{C}e_i\right)\\
&= \mathcal{E}_0\oplus\cdots\oplus
\mathcal{E}_{\ell}\oplus\cdots\oplus\mathcal{E}_n.
\end{aligned}
\end{equation*}
When necessary, we regard the subset $X = \{i_1, i_2, .., i_{\ell}\} \subset [n]$ as the ordered set $i_1<\cdots<i_\ell$.
Set $e_X=e_{i_1}\wedge e_{i_2}\wedge\cdots\wedge e_{i_\ell}$. In particular, set $e_{ij}=e_i\wedge e_j$.
By convention, set $e_{\emptyset}=1$. Let $\partial:\mathcal{E}\to\mathcal{E}$ be the unique morphism such that
$\partial(e_i\wedge e_j)=\partial(e_i)\wedge e_j-e_i\wedge\partial(e_j)$ and $\partial(e_i)=1$, for all $e_i,e_j\in E$ and $n\in\mathbb{N}$.
Let $\Im(\mathcal{M})=\langle\partial(e_C):C\in\mathfrak{C}(\mathcal{M})\rangle$ 
be the two-sided ideal of $\mathcal{E}$ generated by $\{\partial(e_C):C\in\mathfrak{C}(\mathcal{M})\}$.

\begin{definition}
The quotient  $\mathcal{E}/\Im(\mathcal{M})$ is the Orlik--Solomon $\mathbb{C}$-algebra of $\mathcal{M}$, 
denoted by $\mathrm{OS}(\mathcal{M})$.
\end{definition}

%Orlik--Solomon algebras have been extensively studied, see \cite{Falk1993,Falk2001,OrlikSolomon1980,SheltonYuzvinsky1997,Yuzvinsky2001}.
%The residue class in $\mathrm{OS}(\mathcal{M})$ determined by the element $e_X\in\mathcal{E}$ is denoted by $e_X$ as well.
%The algebra $\mathrm{OS}(\mathcal{M})$ has the natural graduation
%\begin{equation*}
%\mathrm{OS}(\mathcal{M})=\mathbb{C}\oplus\mathrm{OS}_1(=\mathcal{E}_1)\oplus\cdots\oplus\mathrm{OS}_{\ell}
%\left(=\mathcal{E}_{\ell}/\Im(\mathcal{M})\cap\mathcal{E}_{\ell}\right)\oplus\cdots\oplus\mathrm{OS}_r.
%\end{equation*}

Orlik--Solomon algebras have been investigated in great depth;
see~\cite{Falk2001, OrlikSolomon1980, Yuzvinsky2001}, among others. For notational convenience, 
the residue class of $e_X\in\mathcal{E}$ in
$\mathrm{OS}(\mathcal{M})$ is still written as $e_X$.
Since $\Im(\mathcal{M})$ is homogeneous,
$\mathrm{OS}(\mathcal{M})$ inherits the grading
\begin{equation*}
\mathrm{OS}(\mathcal{M})
=\bigoplus_{p\geq 0}\mathrm{OS}_p(\mathcal{M}),
\quad
\mathrm{OS}_p(\mathcal{M})
=\mathcal{E}_p/\bigl(\Im(\mathcal{M})\cap\mathcal{E}_p\bigr).
\end{equation*}
Because $\mathcal{M}$ is simple, its homogeneous components
of degrees $0$ and $1$ are naturally identified with
$\mathbb{C}$ and $\mathcal{E}_1$, respectively.
We set $\mathrm{OS}_p(\mathcal{M})=0$ for $p<0$.
We say that $\mathrm{OS}(\mathcal{M})$ determines $T_{\mathcal{M}}(x,y)$ if, for
every finite simple matroid $\mathcal{N}$, an isomorphism
$\mathrm{OS}(\mathcal{M})\cong\mathrm{OS}(\mathcal{N})$ of
$\mathbb{C}$-algebras, without regard to the grading, implies
$T_{\mathcal{M}}(x,y)=T_{\mathcal{N}}(x,y)$.
For the problem of whether the Orlik--Solomon algebra determines the Tutte
polynomial, see~\cite{Falk1997,EschenbrennerFalk1999,CordovilForge2003}.

Falk~\cite{Falk1997} proposed the following conjecture.

\begin{conjecture}(\cite{Falk1997})\label{conj:Falk}
Let $\mathcal{M}$ be a simple matroid of rank three. Then
$\mathrm{OS}(\mathcal{M})$ determines $T_{\mathcal{M}}(x,y)$.
\end{conjecture}

This conjecture had remained unresolved. Cordovil and Forge \cite{CordovilForge2003}
showed that $\mathrm{OS}(\mathcal{M})$ determines $T_{\mathcal{M}}(x,y)$ for
every simple matroid $\mathcal{M}$ of rank at most three that is realizable over
$\mathbb{R}$. By contrast, Eschenbrenner and Falk \cite{EschenbrennerFalk1999}
showed that, for every integer $r\geq 4$, there are simple matroids
$\mathcal{M}$ and $\mathcal{N}$ of rank $r$ whose Orlik--Solomon algebras are
isomorphic as graded $\mathbb{C}$-algebras but whose Tutte polynomials are
distinct. In this paper, we establish the following theorem.

\begin{theorem}\label{thm:main}
Let $\mathcal{M}$ be a simple matroid of rank at most three. Then
$\mathrm{OS}(\mathcal{M})$ determines $T_{\mathcal{M}}(x,y)$.
\end{theorem}

\begin{remark}
Theorem~\ref{thm:main} resolves Conjecture~\ref{conj:Falk}
and extends the result of Cordovil and
Forge~\cite{CordovilForge2003} to all simple matroids
of rank at most three.
\end{remark}

The paper is structured as follows. Section 2 presents the definitions and
lemmas required for the proof of the main theorem. Section 3 establishes that the
Tutte polynomial of a simple matroid of rank at most three is determined by its
Orlik--Solomon $\mathbb{C}$-algebra.

\section{Preliminary lemmas}

We first introduce several definitions and lemmas used in the proof of the
main theorem.

A subset $F\subseteq[n]$ is called a \emph{flat} of $\mathcal{M}$ if
$\operatorname{rk}_{\mathcal{M}}(F\cup\{j\})
>\operatorname{rk}_{\mathcal{M}}(F)$
for every $j\in[n]\setminus F$.
The flats of $\mathcal{M}$, ordered by inclusion, form a lattice,
denoted by $L=L(\mathcal{M})$ and called the \emph{lattice of flats}
of $\mathcal{M}$.
Let $\mu\colon L\times L\to\mathbb{Z}$ denote the M\"obius function of $L$,
defined recursively by
\begin{equation*}
\mu(X,Y)=
\begin{cases}
1, & X=Y,\\[4pt]
-\displaystyle\sum_{\substack{Z\in L\\X\subseteq Z\subsetneq Y}}
\mu(X,Z), & X\subsetneq Y,\\[8pt]
0, & X\nsubseteq Y,
\end{cases}
\end{equation*}
for all $X,Y\in L$.
Since $\mathcal{M}$ is simple, $\varnothing$ is the least element of $L$,
and $\mu(\varnothing,\varnothing)=1$.
%The $p$-th Whitney number of $L$ is defined by
%\begin{equation*}
%w_p(\mathcal{M})
%=(-1)^p
%\sum_{\substack{X\in L\\\operatorname{rk}_{\mathcal{M}}(X)=p}}
%\mu(\varnothing,X).
%\end{equation*}

%\begin{lemma}(\cite{Falk2001})\label{lem:OS-dimension}
%For every integer $p\geq 0$,
%$\dim_{\mathbb{C}}\mathrm{OS}_p(\mathcal{M})=w_p(\mathcal{M})$.
%\end{lemma}

%\begin{lemma}(\cite{Oxley2011})\label{lem:fundamental-circuit-exchange}
%Let $\mathcal{M}$ be a matroid on the ground set $[n]$,
%let $B\in\mathfrak{B}(\mathcal{M})$, and let $c\in[n]\setminus B$.
%Then $B\cup\{c\}$ contains a unique circuit $C$, and $c\in C$.
%Moreover, for every $b\in C\setminus\{c\}$,
%the set $(B\setminus\{b\})\cup\{c\}$ is a basis of $\mathcal{M}$.
%\end{lemma}

\begin{lemma}(\cite{OrlikSolomon1980})\label{lem:OS-Poincare}
Let $\mathcal{M}$ be a simple matroid. Then
\begin{equation*}
\sum_{p\geq 0}\dim_{\mathbb{C}}\mathrm{OS}_p(\mathcal{M})\,t^p=\sum_{X\in L(\mathcal{M})}
\mu(\varnothing,X)(-t)^{\operatorname{rk}_{\mathcal{M}}(X)}.
\end{equation*}
\end{lemma}

By Lemma~\ref{lem:OS-Poincare}, we can get the following corollary.

\begin{cor}\label{lem:rank}
Let $\mathcal{M}$ be a simple matroid of rank $r$
on the ground set $[n]$. Then
\begin{equation*}
n=\dim_{\mathbb{C}}\mathrm{OS}_1(\mathcal{M}),\quad
r=\max\bigl\{p\in\mathbb{Z}_{\geq 0}:\mathrm{OS}_p(\mathcal{M})\neq 0\bigr\}.
\end{equation*}
\end{cor}

\begin{proof}
By Lemma~\ref{lem:OS-Poincare}, comparison of coefficients gives
\begin{equation*}
\dim_{\mathbb{C}}\mathrm{OS}_p(\mathcal{M})
=
(-1)^p
\sum_{\substack{X\in L(\mathcal{M})\\
\operatorname{rk}_{\mathcal{M}}(X)=p}}
\mu(\varnothing,X).
\end{equation*}
Since $\mathcal{M}$ is simple, its rank-one flats are
the singletons, with $\mu(\varnothing,\{i\})=-1$.
Taking $p=1$ therefore yields
$\dim_{\mathbb{C}}\mathrm{OS}_1(\mathcal{M})=n$.

There are no flats of rank greater than $r$, so
$\mathrm{OS}_p(\mathcal{M})=0$ for $p>r$.
The unique rank-$r$ flat is $[n]$, and the M\"obius
function of a geometric lattice satisfies
$(-1)^r\mu(\varnothing,[n])>0$.
Taking $p=r$ therefore gives
$\dim_{\mathbb{C}}\mathrm{OS}_r(\mathcal{M})>0$,
which proves the second identity.
\end{proof}

Let $\mathbb{K}$ be a field.
A unital $\mathbb{K}$-algebra $A$ is said to be a
\emph{nonnegatively graded algebra}, or simply a \emph{graded algebra},
if it can be written as a direct sum of $\mathbb{K}$-vector spaces
\begin{equation*}
A=\bigoplus_{p\geq 0}A_p,
\end{equation*}
with $1_A\in A_0$ and $A_pA_q\subseteq A_{p+q}$ for all $p,q\geq 0$.
Here, $1_A$ is the identity element of $A$, and $A_p$ is referred to as
the \emph{homogeneous component of degree $p$} of $A$.
A nonzero element of $A_p$ is termed a
\emph{homogeneous element of degree $p$}.
Put $\mathbb{K}1_A=\{a1_A:a\in\mathbb{K}\}$.
When $A_0=\mathbb{K}1_A$, the algebra $A$ is called a
\emph{connected graded algebra}.
Under the natural identification of $\mathbb{K}$ with $\mathbb{K}1_A$
given by $a\mapsto a1_A$, this condition is typically written as
$A_0=\mathbb{K}$.

If $A$ is generated by $A_1$ as a unital $\mathbb{K}$-algebra and
$\dim_{\mathbb{K}}A_1<\infty$, then $A$ is said to be
\emph{generated by finitely many homogeneous elements of degree $1$}.
Here, generation by $A_1$ means that every element of $A$ is a finite
$\mathbb{K}$-linear combination of finite products of elements of $A_1$,
with the empty product taken to be $1_A$.
Let $A=\bigoplus_{p\geq 0}A_p$ and $B=\bigoplus_{p\geq 0}B_p$
be two graded $\mathbb{K}$-algebras.
We say that $A$ and $B$ are \emph{isomorphic as ungraded
$\mathbb{K}$-algebras} if there exists a $\mathbb{K}$-algebra
isomorphism $A\to B$.
If there exists a $\mathbb{K}$-algebra isomorphism
$\varphi\colon A\to B$ such that $\varphi(A_p)=B_p$ for every $p\geq 0$,
then $A$ and $B$ are said to be \emph{isomorphic as graded
$\mathbb{K}$-algebras}, and $\varphi$ is called a
\emph{graded $\mathbb{K}$-algebra isomorphism}.

\begin{lemma}(\cite{BellZhang2017})\label{lem:graded}
Suppose $A$ and $B$ are connected graded algebras over a field $\mathbb{K}$,
each generated by finitely many homogeneous elements of degree $1$.
If $A$ and $B$ are isomorphic as ungraded $\mathbb{K}$-algebras,
then $A$ and $B$ are isomorphic as graded $\mathbb{K}$-algebras.
\end{lemma}

%From now on, write $\mathrm{OS}=\mathrm{OS}(\mathcal{M})$ and
%$\mathrm{OS}_p=\mathrm{OS}_p(\mathcal{M})$.
%We continue to use the notation of \cite{CordovilForge2003} and
%identify $\mathrm{OS}_1=\mathcal{E}_1$ naturally with $\mathbb{C}^n$.
%For $\lambda=(\lambda_1,\ldots,\lambda_n)\in\mathbb{C}^n$, set
%$e_\lambda=\sum_{i=1}^{n}\lambda_i e_i$.
%Left multiplication by $e_\lambda$ defines linear maps
%$\mathrm{OS}_p\to\mathrm{OS}_{p+1}$.
%Since $e_\lambda\wedge e_\lambda=0$, these maps form a cochain complex,
%whose $p$-th cohomology group is denoted by
%$H^p(\mathrm{OS},e_\lambda)$.
%The $p$-th \emph{resonance variety} of $\mathrm{OS}$ is defined by
%\begin{equation*}
%\mathcal{R}_p(\mathrm{OS})
%=\bigl\{\lambda\in\mathbb{C}^n:
%H^p(\mathrm{OS},e_\lambda)\neq 0\bigr\}.
%\end{equation*}
Henceforth, write $\mathrm{OS}=\mathrm{OS}(\mathcal{M})$
and $\mathrm{OS}_p=\mathrm{OS}_p(\mathcal{M})$.
Identify $\mathrm{OS}_1$ naturally with $\mathbb{C}^n$.
For $\lambda=(\lambda_1,\ldots,\lambda_n)\in\mathbb{C}^n$,
put $e_\lambda=\sum_{i=1}^{n}\lambda_i e_i$.
For each integer $p\geq 0$, define
$d_\lambda^p:\mathrm{OS}_p\to\mathrm{OS}_{p+1}$
by $d_\lambda^p(a)=e_\lambda a$.
Because $e_\lambda^2=0$, these maps constitute a cochain complex.
Its $p$-th cohomology vector space is
\begin{equation*}
H^p(\mathrm{OS},e_\lambda)
=\ker d_\lambda^p/\operatorname{im}d_\lambda^{p-1},
\end{equation*}
where $d_\lambda^{-1}$ is the zero map from
$\mathrm{OS}_{-1}=0$ to $\mathrm{OS}_0$.
For every integer $p\geq 0$, define the
$p$-th \emph{resonance variety} by
\begin{equation*}
\mathcal{R}_p(\mathrm{OS})
=\{\lambda\in\mathbb{C}^n:
 H^p(\mathrm{OS},e_\lambda)\neq 0\}.
\end{equation*}
Let
$L_2(\mathcal{M})
=\bigl\{X\in L(\mathcal{M}):
\operatorname{rk}_{\mathcal{M}}(X)=2\bigr\}$
and
$\Im_2(\mathcal{M})=\Im(\mathcal{M})\cap\mathcal{E}_2$.
For each $X\in L_2(\mathcal{M})$, define
$\mathcal{E}_2(X)
=\operatorname{span}_{\mathbb{C}}
\{e_{ij}:i,j\in X,\ i<j\}$,
where $e_{ij}=e_i\wedge e_j$.
Set $\Im_2(X)=\Im_2(\mathcal{M})\cap\mathcal{E}_2(X)$.
For each $X\in L_2(\mathcal{M})$, let
$L_X=\{\lambda\in\mathbb{C}^n:
\lambda_i=0\text{ for all }i\notin X,\,
\sum_{i\in X}\lambda_i=0\}$.
%When $|X|\geq 3$, this space is called a \emph{local subspace},
%and $L_X$ is an irreducible component of
%$\mathcal{R}_1(\mathrm{OS})$, called a \emph{local component}.
%The remaining irreducible components are called
%\emph{nonlocal components}.
%For any linear subspace $U\subseteq\mathbb{C}^n$, set
%\begin{equation*}
%\bigwedge\nolimits^2 U
%=\operatorname{span}_{\mathbb{C}}
%\{e_\lambda\wedge e_\mu:\lambda,\mu\in U\}
%\subseteq\mathcal{E}_2.
%\end{equation*}
% All exterior products below are taken in the exterior algebra $\mathcal{E}$.
When $|X|\geq 3$, the subspace $L_X$ is an irreducible
component of $\mathcal{R}_1(\mathrm{OS})$, called a
\emph{local component}.
The positive-dimensional irreducible components that are not
of this form are called \emph{nonlocal components}.
For every linear subspace $U\subseteq\mathbb{C}^n$, define
\begin{equation*}
\bigwedge\nolimits^2 U
=\operatorname{span}_{\mathbb{C}}
 \{e_\lambda\wedge e_\mu:\lambda,\mu\in U\}
\subseteq\mathcal{E}_2.
\end{equation*}
Here and in the definitions of $F_d(\mathrm{OS})$ below,
exterior products are taken in $\mathcal{E}$.

For $S\subseteq[n]$, define
$\overline{S}
=\left\{i\in[n]:
\operatorname{rk}_{\mathcal{M}}(S\cup\{i\})
=\operatorname{rk}_{\mathcal{M}}(S)\right\}$,
which is called the \emph{closure} of $S$ in $\mathcal{M}$.
It is the smallest flat containing $S$, and
$\operatorname{rk}_{\mathcal{M}}(\overline{S})
=\operatorname{rk}_{\mathcal{M}}(S)$.

\begin{lemma}(\cite{OrlikSolomon1980})\label{lem:OS-flat-decomposition}
Let $\mathcal{M}$ be a simple matroid on $[n]$.
For each flat $X\in L(\mathcal{M})$ of rank $p$, set
\begin{equation*}
\mathrm{OS}_p(X)=\operatorname{span}_{\mathbb{C}}\bigl\{e_S\in\mathrm{OS}_p(\mathcal{M}):
S\subseteq X,\ |S|=p\bigr\},
\end{equation*}
where $e_S$ denotes the residue class of the
corresponding exterior monomial. Then
\begin{equation*}
\mathrm{OS}_p(\mathcal{M})=\bigoplus_{\substack{X\in L(\mathcal{M})\\
\operatorname{rk}_{\mathcal{M}}(X)=p}}\mathrm{OS}_p(X).
\end{equation*}
\end{lemma}

Applying Lemma~\ref{lem:OS-flat-decomposition},
we obtain the following corollaries.

\begin{cor}\label{lem:local-decomposition}
Let $\mathcal{M}$ be a simple matroid on the ground set $[n]$. Then
\begin{equation*}
\mathcal{E}_2=\bigoplus_{X\in L_2(\mathcal{M})}\mathcal{E}_2(X),\qquad
\Im_2(\mathcal{M})=\bigoplus_{X\in L_2(\mathcal{M})}\Im_2(X).
\end{equation*}
Moreover, for each $X\in L_2(\mathcal{M})$, we have
\begin{equation*}
\Im_2(X)=\operatorname{span}_{\mathbb{C}}\bigl\{\partial(e_C):
C\subseteq X,\ |C|=3\bigr\}.
\end{equation*}
\end{cor}

\begin{proof}
Since $\mathcal{M}$ is simple, each two-element subset
$\{i,j\}\subseteq[n]$ is contained in the unique rank-two
flat $\overline{\{i,j\}}$.
Hence the standard basis vectors $e_{ij}$ are partitioned
by the rank-two flats, giving
\begin{equation*}
\mathcal{E}_2
=
\bigoplus_{X\in L_2(\mathcal{M})}\mathcal{E}_2(X).
\end{equation*}
By Lemma~\ref{lem:OS-flat-decomposition}, we have
\begin{equation*}
\mathrm{OS}_2(\mathcal{M})
=
\bigoplus_{X\in L_2(\mathcal{M})}\mathrm{OS}_2(X).
\end{equation*}

For $u\in\Im_2(\mathcal{M})$, write uniquely
$u=\sum_X u_X$, where $u_X\in\mathcal{E}_2(X)$.
The residue class of $u$ is zero, and the residue class
of each $u_X$ belongs to $\mathrm{OS}_2(X)$.
The latter direct sum decomposition therefore implies
that each $u_X$ has zero residue class.
Thus
\begin{equation*}
u_X\in
\Im_2(\mathcal{M})\cap\mathcal{E}_2(X)
=
\Im_2(X).
\end{equation*}
Since each $\Im_2(X)$ is contained in $\Im_2(\mathcal{M})$,
and the subspaces $\mathcal{E}_2(X)$ form a direct sum,
we obtain
\begin{equation*}
\Im_2(\mathcal{M})
=
\bigoplus_{X\in L_2(\mathcal{M})}\Im_2(X).
\end{equation*}

Since $\mathcal{M}$ is simple, $\Im_2(\mathcal{M})$
is spanned by the boundaries of three-element circuits.
Each such circuit $C$ lies in the unique rank-two flat
$\overline{C}$, with
$\partial(e_C)\in\mathcal{E}_2(\overline{C})$.
The direct sum decomposition of $\mathcal{E}_2$ therefore
implies that $\Im_2(X)$ is spanned by those boundaries
with $\overline{C}=X$.
Since the three-element subsets of $X$ are precisely
the three-element circuits with closure $X$, we obtain
\begin{equation*}
\Im_2(X)=
\operatorname{span}_{\mathbb{C}}
\bigl\{\partial(e_C):C\subseteq X,\ |C|=3\bigr\}.
\end{equation*}
\end{proof}

\begin{cor}\label{lem:local-relations}
Let $\mathcal{M}$ be a simple matroid, and let
$X\in L_2(\mathcal{M})$.
Then
\begin{equation*}
\Im_2(X)=\bigwedge\nolimits^2 L_X,
\qquad
\dim_{\mathbb{C}}L_X=|X|-1,
\qquad
\dim_{\mathbb{C}}\Im_2(X)=\binom{|X|-1}{2}.
\end{equation*}
In particular, if $|X|=2$, then $\Im_2(X)=0$.
\end{cor}

\begin{proof}
Every pair of distinct elements of $\mathcal{M}$ lies in
a unique rank-two flat, so
\begin{equation*}
\mathcal{E}_2
=
\bigoplus_{Y\in L_2(\mathcal{M})}\mathcal{E}_2(Y).
\end{equation*}
By Lemma~\ref{lem:OS-flat-decomposition}, the corresponding
residue classes lie in the direct summands
$\mathrm{OS}_2(Y)$.
Thus a sum of elements from the $\mathcal{E}_2(Y)$
belongs to $\Im_2(\mathcal{M})$ if and only if each
summand does. Hence
\begin{equation*}
\Im_2(\mathcal{M})
=
\bigoplus_{Y\in L_2(\mathcal{M})}\Im_2(Y).
\end{equation*}
Since $\mathcal{M}$ is simple, $\Im_2(\mathcal{M})$
is spanned by three-element circuit boundaries.
Grouping these boundaries by their unique rank-two closures,
and noting that every three-element subset of $X$ is a circuit
with closure $X$, gives
\begin{equation*}
\Im_2(X)
=
\operatorname{span}_{\mathbb{C}}
\bigl\{\partial(e_C):C\subseteq X,\ |C|=3\bigr\}.
\end{equation*}

Fix $i_0\in X$.
The vectors $e_i-e_{i_0}$, for $i\in X\setminus\{i_0\}$,
form a basis of $L_X$.
For distinct $i,j\in X\setminus\{i_0\}$,
\begin{equation*}
(e_i-e_{i_0})\wedge(e_j-e_{i_0})
=
\partial(e_{i_0}\wedge e_i\wedge e_j)
\in\Im_2(X),
\end{equation*}
so $\bigwedge^2L_X\subseteq\Im_2(X)$.
Conversely, for distinct $a,b,c\in X$,
\begin{equation*}
\partial(e_a\wedge e_b\wedge e_c)
=
(e_b-e_a)\wedge(e_c-e_a)
\in\bigwedge\nolimits^2L_X.
\end{equation*}
Therefore $\Im_2(X)=\bigwedge^2L_X$, and
\begin{equation*}
\dim_{\mathbb{C}}L_X=|X|-1,
\qquad
\dim_{\mathbb{C}}\Im_2(X)=\binom{|X|-1}{2}.
\end{equation*}
In particular, $\Im_2(X)=0$ if $|X|=2$.
\end{proof}

For a linear subspace $U\subseteq\mathbb{C}^n$, define its
\emph{support} by
\begin{equation*}
\operatorname{supp}(U)
=\bigl\{i\in[n]:
\text{there exists }\lambda\in U
\text{ such that }\lambda_i\neq 0\bigr\}.
\end{equation*}
For $S\subseteq[n]$, let $\mathcal{M}|S$ denote the restriction of
$\mathcal{M}$ to $S$, that is, the matroid on the ground set $S$
whose independent sets are precisely the independent sets of
$\mathcal{M}$ contained in $S$.
Let $L_2(\mathcal{M}|S)$ denote the set of all flats of rank $2$
in $\mathcal{M}|S$.
A partition $\Pi=\{\pi_1,\ldots,\pi_k\}$ of $S$ is called a
\emph{$k$-coloring} of $\mathcal{M}|S$, and each $\pi_\alpha$ is
called a \emph{color class}.
The coloring is called \emph{regular} if, for every
$Y\in L_2(\mathcal{M}|S)$, we have
\begin{equation*}
\bigl|\{\pi\in\Pi:\pi\cap Y\neq\varnothing\}\bigr|\in\{1,k\}.
\end{equation*}
Given the above partition and positive integers $m_i$ $(i\in S)$,
we use the natural identification $\mathbb{C}^n=\mathcal{E}_1$ and set
$u_\alpha=\sum_{i\in\pi_\alpha}m_i e_i$ for $1\leq\alpha\leq k$.

\begin{lemma}(\cite{FalkYuzvinsky2007,LibgoberYuzvinsky2000})
\label{lem:resonance-structure}
Let $\mathcal{M}$ be a simple matroid of rank $3$ on the ground set
$[n]$, and consider its Orlik--Solomon algebra over $\mathbb{C}$.
Then every positive-dimensional irreducible component $W$ of
$\mathcal{R}_1(\mathrm{OS})$ is a linear subspace, and
\begin{equation*}
h=\dim_{\mathbb{C}}W\geq 2,
\qquad
\bigwedge\nolimits^2 W\subseteq\Im_2(\mathcal{M}).
\end{equation*}
Let $S=\operatorname{supp}(W)$ and $k=h+1$.
Then there exist a regular $k$-coloring
$\Pi=\{\pi_1,\ldots,\pi_k\}$ of $\mathcal{M}|S$,
a positive integer $d$, and positive integers $m_i$ $(i\in S)$
such that $\sum_{i\in\pi_\alpha}m_i=d$ for $1\leq\alpha\leq k$,
and the corresponding vectors $u_1,\ldots,u_k$ satisfy
\begin{equation*}
W=\left\{
\sum_{\alpha=1}^{k}t_\alpha u_\alpha:
t_1,\ldots,t_k\in\mathbb{C},\,
\sum_{\alpha=1}^{k}t_\alpha=0
\right\}.
\end{equation*}
Moreover, for every $Y\in L_2(\mathcal{M}|S)$,
one of the following holds:
\begin{enumerate}
\item[\textup{(i)}]
$Y$ is contained in some color class $\pi_\alpha$;
\item[\textup{(ii)}]
$Y$ meets all $k$ color classes, and
\begin{equation*}
\sum_{i\in Y\cap\pi_1}m_i
=\cdots
=\sum_{i\in Y\cap\pi_k}m_i.
\end{equation*}
\end{enumerate}
\end{lemma}

For a simple matroid $\mathcal{M}$ of rank $3$, denote by
$\mathscr{W}(\mathrm{OS})$ the set of all
positive-dimensional irreducible components of
$\mathcal{R}_1(\mathrm{OS})$.
According to Lemma~\ref{lem:resonance-structure}, every
$W\in\mathscr{W}(\mathrm{OS})$ is a linear subspace
satisfying $\dim_{\mathbb{C}}W\geq 2$.
For each $X\in L_2(\mathcal{M})$, define the coordinate projection
$\rho_X\colon\mathbb{C}^n\to\mathbb{C}^n$ by
\begin{equation*}
(\rho_X(\lambda))_i=
\begin{cases}
\lambda_i, & i\in X,\\
0, & i\notin X,
\end{cases}
\end{equation*}
where $\lambda\in\mathbb{C}^n$ and $i\in[n]$.
Under the natural identification of $\mathbb{C}^n$ with
$\mathcal{E}_1$ via $\lambda\mapsto e_\lambda$,
we use $\rho_X$ also to denote the corresponding linear map, namely,
$\rho_X(e_\lambda)
=e_{\rho_X(\lambda)}
=\sum_{i\in X}\lambda_i e_i$.
Using the direct sum decomposition in
Corollary~\ref{lem:local-decomposition}, define the projection
$\pi_X\colon\mathcal{E}_2\to\mathcal{E}_2(X)$ onto the direct summand
$\mathcal{E}_2(X)$ by
\begin{equation*}
\pi_X\left(\sum_{1\leq i<j\leq n}a_{ij}e_{ij}\right)
=\sum_{\substack{i,j\in X\\i<j}}a_{ij}e_{ij},
\end{equation*}
where $a_{ij}\in\mathbb{C}$.
This projection satisfies
$\pi_X(\Im_2(\mathcal{M}))=\Im_2(X)$.
By definition, for any $\lambda,\mu\in\mathbb{C}^n$, we have
\begin{equation*}
\pi_X(e_\lambda\wedge e_\mu)
=e_{\rho_X(\lambda)}\wedge e_{\rho_X(\mu)}.
\end{equation*}

\begin{lemma}\label{lem:projection}
Let $\mathcal{M}$ be a simple matroid of rank 3
on $[n]$, and let $W\in\mathscr{W}(\mathrm{OS})$.
Set $h=\dim_{\mathbb{C}}W\geq 2$.
Then, for each $X\in L_2(\mathcal{M})$,
exactly one of the following holds:
\begin{enumerate}
\item[{\rm(i)}]
$\dim_{\mathbb{C}}\rho_X(W)\leq 1$ and
$\pi_X(\bigwedge^2 W)=\{0\}$;

\item[{\rm(ii)}]
$\rho_X(W)\subseteq L_X$ and
$\rho_X|_W$ is injective.
Moreover, $\pi_X|_{\bigwedge^2 W}$ is injective,
and consequently
\begin{equation*}
\dim_{\mathbb{C}}
\pi_X\bigl(\textstyle\bigwedge^2 W\bigr)
=
\binom{h}{2},
\qquad
|X|\geq h+1.
\end{equation*}
\end{enumerate}
In particular, if $|X|\leq h$, then
$\pi_X(\bigwedge^2 W)=\{0\}$.
Consequently,
\begin{equation*}%\label{eq:large-flat-support}
\bigwedge^2 W
\subseteq
\bigoplus_{\substack{X\in L_2(\mathcal{M})\\
|X|\geq h+1}}
\Im_2(X).
\end{equation*}
\end{lemma}

\begin{proof}
Fix $X\in L_2(\mathcal{M})$.
By the definitions of $\rho_X$ and $\pi_X$, for
any $\lambda,\mu\in W$, we have
\begin{equation*}\label{eq:ex}
\pi_X(e_\lambda\wedge e_\mu)
=
e_{\rho_X(\lambda)}\wedge e_{\rho_X(\mu)}.
\end{equation*}
Under the identification
$\mathbb{C}^n=\mathcal{E}_1$, this means that
$\pi_X|_{\bigwedge^2 W}$ is the exterior square
of $\rho_X|_W$.

Set $S=\operatorname{supp}(W)$ and $k=h+1$.
By Lemma~\ref{lem:resonance-structure}, there exist
a regular $k$-coloring
$\Pi=\{\pi_1,\ldots,\pi_k\}$ of $\mathcal{M}|S$
and positive integers $m_i$ $(i\in S)$ such that
\begin{equation*}
W=
\left\{
\sum_{\alpha=1}^{k}t_\alpha u_\alpha:
t_1,\ldots,t_k\in\mathbb{C},\
\sum_{\alpha=1}^{k}t_\alpha=0
\right\},
\end{equation*}
where for $1\leq\alpha\leq k$, $u_\alpha=\sum_{i\in\pi_\alpha}m_i e_i$.

Put $Y=X\cap S$.
If $|Y|\leq 1$, then every vector in $\rho_X(W)$
is supported on $Y$.
Hence $\dim_{\mathbb{C}}\rho_X(W)\leq 1$, and
the formula \eqref{eq:ex} gives $\pi_X(\bigwedge^2 W)=\{0\}$.
Thus {\rm(i)} holds, whereas {\rm(ii)} is impossible
because $h\geq 2$.

Suppose that $|Y|\geq 2$.
Since $X$ is a flat of $\mathcal{M}$,
the intersection $Y=X\cap S$ is a flat of
the restriction $\mathcal{M}|S$.
Moreover, simplicity implies that any two
distinct elements of $Y$ are independent.
Since $Y\subseteq X$ and
$\operatorname{rk}_{\mathcal{M}}(X)=2$, we obtain
\begin{equation*}
2
\leq
\operatorname{rk}_{\mathcal{M}|S}(Y)
=
\operatorname{rk}_{\mathcal{M}}(Y)
\leq 2.
\end{equation*}
Thus $Y\in L_2(\mathcal{M}|S)$, and
Lemma~\ref{lem:resonance-structure} applies to $Y$.

If $Y$ is contained in some color class
$\pi_\alpha$, then
\begin{equation*}
\rho_X(W)\subseteq
\operatorname{span}_{\mathbb{C}}\{\rho_X(u_\alpha)\}.
\end{equation*}
Thus $\dim_{\mathbb{C}}\rho_X(W)\leq 1$, and
the formula \eqref{eq:ex} gives
$\pi_X(\bigwedge^2 W)=\{0\}$.
Hence {\rm(i)} holds.

Otherwise, $Y$ meets every color class and
\begin{equation*}
\sum_{i\in Y\cap\pi_1}m_i=\cdots=\sum_{i\in Y\cap\pi_k}m_i=s_X>0.
\end{equation*}
For $1\leq\alpha\leq k$, let
\begin{equation*}
v_\alpha=\rho_X(u_\alpha)=\sum_{i\in Y\cap\pi_\alpha}m_i e_i.
\end{equation*}
These vectors are nonzero and have pairwise
disjoint supports, so they are linearly independent.

To prove that $\rho_X|_W$ is injective, take
$w=\sum_{\alpha=1}^{k}t_\alpha u_\alpha\in W$
such that $\rho_X(w)=0$.
Then
\begin{equation*}
0=\rho_X(w)
=\sum_{\alpha=1}^{k}t_\alpha v_\alpha.
\end{equation*}
The linear independence of the $v_\alpha$
implies that every $t_\alpha$ is zero.
Therefore $w=0$, proving injectivity.

For any $\lambda\in W$, write
$e_\lambda=\sum_{\alpha=1}^{k}t_\alpha u_\alpha$
with $\sum_{\alpha=1}^{k}t_\alpha=0$.
Then
\begin{equation*}
\sum_{i\in X}(\rho_X(\lambda))_i
=
\sum_{\alpha=1}^{k}
t_\alpha
\sum_{i\in Y\cap\pi_\alpha}m_i
=
s_X\sum_{\alpha=1}^{k}t_\alpha
=
0.
\end{equation*}
Since $\rho_X(\lambda)$ is supported on $X$,
the definition of $L_X$ yields
$\rho_X(W)\subseteq L_X$.
Consequently,
\begin{equation*}
h
=
\dim_{\mathbb{C}}\rho_X(W)
\leq
\dim_{\mathbb{C}}L_X
=
|X|-1,
\end{equation*}
and hence $|X|\geq h+1$.

The exterior square of an injective linear map
is injective. Therefore
$\pi_X|_{\bigwedge^2 W}$ is injective, and
\begin{equation*}
\dim_{\mathbb{C}}
\pi_X\bigl(\textstyle\bigwedge^2 W\bigr)
=
\dim_{\mathbb{C}}\bigwedge^2 W
=
\binom{h}{2}.
\end{equation*}
This proves {\rm(ii)}.
Since $h\geq 2$, alternative {\rm(ii)} implies
$\dim_{\mathbb{C}}\rho_X(W)=h\geq 2$.
Thus the two alternatives are mutually exclusive.

In particular, if $|X|\leq h$, then {\rm(ii)}
is impossible, so
$\pi_X(\bigwedge^2 W)=\{0\}$.
By Lemma~\ref{lem:resonance-structure},
$\bigwedge^2 W\subseteq\Im_2(\mathcal{M})$.
The direct sum decomposition in
Corollary~\ref{lem:local-decomposition},
together with
$\pi_X(\Im_2(\mathcal{M}))=\Im_2(X)$,
therefore gives
\begin{equation*}
\bigwedge^2 W
\subseteq
\bigoplus_{\substack{X\in L_2(\mathcal{M})\\
|X|\geq h+1}}
\Im_2(X).
\end{equation*}
\end{proof}

Let $\mathcal{M}$ be a simple matroid of rank $3$ on the ground set $[n]$.
For each integer $p\geq 2$, let
$c_p(\mathcal{M})
=\bigl|\{X\in L_2(\mathcal{M}):|X|=p\}\bigr|$
denote the number of flats of rank $2$ in $\mathcal{M}$ with $p$ elements.
For each integer $d\geq 2$, define
\begin{equation*}
F_d(\mathrm{OS})
=\sum_{\substack{W\in\mathscr{W}(\mathrm{OS})\\
\dim_{\mathbb{C}}W\geq d}}
\bigwedge\nolimits^2 W.
\end{equation*}
According to Lemma~\ref{lem:resonance-structure}, we have
$\bigwedge\nolimits^2 W\subseteq\Im_2(\mathcal{M})$
for every $W\in\mathscr{W}(\mathrm{OS})$.
Together with the preceding definition, this gives
$F_{d+1}(\mathrm{OS})\subseteq F_d(\mathrm{OS})
\subseteq\Im_2(\mathcal{M})$
for every integer $d\geq 2$.

\begin{lemma}(\cite{Falk1997})\label{lem:local-resonance}
Let $\mathcal{M}$ be a simple matroid.
Suppose $X\in L_2(\mathcal{M})$ and $|X|\geq 3$. Then $L_X$ is an
irreducible component of $\mathcal{R}_1(\mathrm{OS})$ of dimension
$|X|-1$, referred to as a local component.
\end{lemma}

\begin{lemma}\label{lem:filtration}
Let $\mathcal{M}$ be a simple matroid of rank $3$ on the ground set $[n]$.
Then, for every integer $d\geq 2$, we have
\begin{equation}\label{eq:filtration-identity}
F_d(\mathrm{OS})
=\bigoplus_{\substack{X\in L_2(\mathcal{M})\\|X|\geq d+1}}
\Im_2(X),
\end{equation}
and
\begin{equation}\label{eq:filtration-dimension}
\dim_{\mathbb{C}}F_d(\mathrm{OS})
=\sum_{p\geq d+1}\binom{p-1}{2}c_p(\mathcal{M}).
\end{equation}
\end{lemma}

\begin{proof}
Fix an integer $d\geq 2$.
We begin by proving the formula \eqref{eq:filtration-identity}.
Let $W\in\mathscr{W}(\mathrm{OS})$ satisfy
$h=\dim_{\mathbb{C}}W\geq d$.
According to Lemma~\ref{lem:projection}, we have
\begin{equation*}
\bigwedge\nolimits^2 W
\subseteq
\bigoplus_{\substack{X\in L_2(\mathcal{M})\\|X|\geq h+1}}
\Im_2(X)
\subseteq
\bigoplus_{\substack{X\in L_2(\mathcal{M})\\|X|\geq d+1}}
\Im_2(X).
\end{equation*}
By summing over all such $W$ and invoking the definition
of $F_d(\mathrm{OS})$, we obtain
\begin{equation*}
F_d(\mathrm{OS})
\subseteq
\bigoplus_{\substack{X\in L_2(\mathcal{M})\\|X|\geq d+1}}
\Im_2(X).
\end{equation*}

Conversely, let $X\in L_2(\mathcal{M})$ satisfy $|X|\geq d+1$.
Since $d\geq 2$, we have $|X|\geq 3$.
According to Lemma~\ref{lem:local-resonance},
we have $L_X\in\mathscr{W}(\mathrm{OS})$ and
$\dim_{\mathbb{C}}L_X=|X|-1\geq d$.
Therefore, the definition of $F_d(\mathrm{OS})$ together with
Corollary~\ref{lem:local-relations} yields
\begin{equation*}
\Im_2(X)=\bigwedge\nolimits^2 L_X\subseteq F_d(\mathrm{OS}).
\end{equation*}
It follows that $\Im_2(X)\subseteq F_d(\mathrm{OS})$ whenever
$X\in L_2(\mathcal{M})$ and $|X|\geq d+1$.
As $F_d(\mathrm{OS})$ is a linear subspace, it contains the sum
of these subspaces.
Moreover, Corollary~\ref{lem:local-decomposition} shows that this sum is direct.
Thus,
\begin{equation*}
\bigoplus_{\substack{X\in L_2(\mathcal{M})\\|X|\geq d+1}}
\Im_2(X)
=
\sum_{\substack{X\in L_2(\mathcal{M})\\|X|\geq d+1}}
\Im_2(X)
\subseteq F_d(\mathrm{OS}).
\end{equation*}
This proves the reverse inclusion and establishes the identity
\eqref{eq:filtration-identity}.

According to Corollary~\ref{lem:local-decomposition}, the subspaces $\Im_2(X)$
associated with distinct flats constitute a direct sum.
Using this together with Corollary~\ref{lem:local-relations} and the
definition of $c_p(\mathcal{M})$, we obtain
\begin{equation*}
\begin{aligned}
\dim_{\mathbb{C}}F_d(\mathrm{OS})
&=\sum_{\substack{X\in L_2(\mathcal{M})\\|X|\geq d+1}}
  \dim_{\mathbb{C}}\Im_2(X)\\
&=\sum_{\substack{X\in L_2(\mathcal{M})\\|X|\geq d+1}}
  \binom{|X|-1}{2}\\
&=\sum_{p\geq d+1}\binom{p-1}{2}c_p(\mathcal{M}).
\end{aligned}
\end{equation*}
This proves the formula \eqref{eq:filtration-dimension}.
\end{proof}

\begin{lemma}\label{lem:recovery}
Let $\mathcal{M}$ be a simple matroid of rank $3$ on the ground set $[n]$.
Then $n=\dim_{\mathbb{C}}\mathrm{OS}_1$ and, for every integer $p\geq 3$,
\begin{equation}\label{eq:recover}
c_p(\mathcal{M})
=\frac{
\dim_{\mathbb{C}}F_{p-1}(\mathrm{OS})
-\dim_{\mathbb{C}}F_p(\mathrm{OS})
}{
\binom{p-1}{2}
}.
\end{equation}
Moreover,
\begin{equation}\label{eq:recover-two}
c_2(\mathcal{M})
=\binom{n}{2}
-\sum_{p=3}^{n-1}\binom{p}{2}c_p(\mathcal{M}).
\end{equation}
If two simple matroids of rank $3$ have isomorphic graded
Orlik--Solomon $\mathbb{C}$-algebras, then their values of $c_p$
are equal for every integer $p\geq 2$.
\end{lemma}

\begin{proof}
We begin by proving the two recovery formulas.
According to Corollary~\ref{lem:rank}, we have
$n=\dim_{\mathbb{C}}\mathrm{OS}_1$.
Since $\mathcal{M}$ has rank $3$, every $X\in L_2(\mathcal{M})$
is a proper subset of $[n]$, and hence $|X|\leq n-1$.
Thus, $c_s(\mathcal{M})=0$ for $s\geq n$, so all sums below
contain only finitely many nonzero terms.
Fix an integer $p\geq 3$.
Since $p-1\geq 2$, applying the formula
\eqref{eq:filtration-dimension} in Lemma~\ref{lem:filtration}
with $d=p-1$ and $d=p$, respectively, yields
\begin{equation*}
\begin{aligned}
&\dim_{\mathbb{C}}F_{p-1}(\mathrm{OS})=\sum_{s\geq p}\binom{s-1}{2}c_s(\mathcal{M}),\\
&\dim_{\mathbb{C}}F_p(\mathrm{OS})=\sum_{s\geq p+1}\binom{s-1}{2}c_s(\mathcal{M}).
\end{aligned}
\end{equation*}
Subtracting these formulas yields
\begin{equation*}
\dim_{\mathbb{C}}F_{p-1}(\mathrm{OS})
-\dim_{\mathbb{C}}F_p(\mathrm{OS})
=\binom{p-1}{2}c_p(\mathcal{M}).
\end{equation*}
Since $\binom{p-1}{2}>0$, dividing both sides by
$\binom{p-1}{2}$ proves the formula \eqref{eq:recover}.

We now turn to the proof of formula \eqref{eq:recover-two}.
Since $\mathcal{M}$ is simple, every two-element subset is of rank $2$,
and its closure is the unique flat of rank $2$ that contains it.
Consequently, the two-element subsets of $[n]$ can be grouped
according to their closures.
For each $X\in L_2(\mathcal{M})$, the resulting class consists
precisely of the two-element subsets of $X$ and therefore has
cardinality $\binom{|X|}{2}$.
By the definition of $c_p(\mathcal{M})$ and the bounds
$2\leq |X|\leq n-1$, we deduce
\begin{equation*}
\begin{aligned}
\binom{n}{2}
&=\sum_{X\in L_2(\mathcal{M})}\binom{|X|}{2}\\
&=\sum_{p=2}^{n-1}\binom{p}{2}c_p(\mathcal{M})\\
&=c_2(\mathcal{M})
  +\sum_{p=3}^{n-1}\binom{p}{2}c_p(\mathcal{M}).
\end{aligned}
\end{equation*}
Rearranging gives the formula \eqref{eq:recover-two}.

We conclude by proving the claim on graded isomorphisms.
Let $\mathcal{N}$ be a simple matroid of rank $3$ on the ground set
$[m]$, and assume that
$\varphi\colon\mathrm{OS}(\mathcal{M})\to\mathrm{OS}(\mathcal{N})$
is a graded $\mathbb{C}$-algebra isomorphism.
Let $g=\varphi|_{\mathrm{OS}_1(\mathcal{M})}$.
With the natural identifications of the homogeneous components of degree $1$
and the corresponding coordinate spaces, we treat $g$ as a linear
isomorphism from $\mathbb{C}^n$ to $\mathbb{C}^m$.
Hence $n=m$ and $\varphi(e_\lambda)=e_{g(\lambda)}$.
For all $\lambda\in\mathbb{C}^n$, $q\geq 0$, and
$a\in\mathrm{OS}_q(\mathcal{M})$, we have
\begin{equation*}
\varphi(e_\lambda a)
=\varphi(e_\lambda)\varphi(a)
=e_{g(\lambda)}\varphi(a).
\end{equation*}
It follows that $\varphi$ induces an isomorphism between the corresponding
cochain complexes, so
\begin{equation*}
H^1(\mathrm{OS}(\mathcal{M}),e_\lambda)
\cong
H^1(\mathrm{OS}(\mathcal{N}),e_{g(\lambda)}).
\end{equation*}
According to the definition of the resonance varieties, we obtain
\begin{equation*}
g\bigl(\mathcal{R}_1(\mathrm{OS}(\mathcal{M}))\bigr)
=\mathcal{R}_1(\mathrm{OS}(\mathcal{N})).
\end{equation*}
Because $g$ and its inverse are both linear maps, the restriction of $g$
gives an isomorphism of algebraic varieties between these
resonance varieties.
Therefore, the map $W\mapsto g(W)$ provides a
dimension-preserving bijection between
$\mathscr{W}(\mathrm{OS}(\mathcal{M}))$ and
$\mathscr{W}(\mathrm{OS}(\mathcal{N}))$.
Denote by
$\bigwedge\nolimits^2 g:
 \bigwedge\nolimits^2\mathbb{C}^n
 \to\bigwedge\nolimits^2\mathbb{C}^m$
the induced linear map, defined by
\begin{equation*}
(\bigwedge\nolimits^2 g)(\lambda\wedge\mu)
=g(\lambda)\wedge g(\mu)
\quad(\lambda,\mu\in\mathbb{C}^n).
\end{equation*}
For every $W\in\mathscr{W}(\mathrm{OS}(\mathcal{M}))$,
we have
\begin{equation*}
(\bigwedge\nolimits^2 g)(\bigwedge\nolimits^2 W)
=\bigwedge\nolimits^2(g(W)).
\end{equation*}
Combining this dimension-preserving bijection with the definition of $F_d$,
we deduce, for any integer $d\geq 2$,
\begin{equation*}
\begin{aligned}
(\bigwedge\nolimits^2 g)
\bigl(F_d(\mathrm{OS}(\mathcal{M}))\bigr)
&=\sum_{\substack{
W\in\mathscr{W}(\mathrm{OS}(\mathcal{M}))\\
\dim_{\mathbb{C}}W\geq d}}
\bigwedge\nolimits^2(g(W))\\
&=F_d(\mathrm{OS}(\mathcal{N})).
\end{aligned}
\end{equation*}
Since $\bigwedge\nolimits^2 g$ is a linear isomorphism, it follows that
\begin{equation*}
\dim_{\mathbb{C}}F_d(\mathrm{OS}(\mathcal{M}))
=\dim_{\mathbb{C}}F_d(\mathrm{OS}(\mathcal{N}))
\end{equation*}
for every $d\geq 2$.
Together with $n=m$, the formulas \eqref{eq:recover} and
\eqref{eq:recover-two} now imply that
$c_p(\mathcal{M})=c_p(\mathcal{N})$ for every $p\geq 2$.
\end{proof}

To express the Tutte polynomial in terms of the numbers of flats,
we next give a reconstruction formula for matroids of rank $3$.
For integers $r,n$ with $0\leq r\leq n$, let $U_{r,n}$ denote
the uniform matroid of rank $r$ on the ground set $[n]$,
whose independent sets are precisely the subsets of $[n]$
of cardinality at most $r$.
Thus, for every $A\subseteq[n]$, we have
$\operatorname{rk}_{U_{r,n}}(A)=\min\{|A|,r\}$.

\begin{lemma}\label{lem:Tutte}
Let $\mathcal{M}$ be a simple matroid of rank $3$ on the ground set $[n]$.
Then
\begin{equation}\label{eq:tutte-rankthree}
T_{\mathcal{M}}(x,y)
=T_{U_{3,n}}(x,y)
+(xy-x-y)\sum_{p=3}^{n-1}c_p(\mathcal{M})
\sum_{j=3}^{p}\binom{p}{j}(y-1)^{j-3},
\end{equation}
where
\begin{equation}\label{eq:uniform-three}
T_{U_{3,n}}(x,y)
=x^3+(n-3)x^2+\binom{n-2}{2}x
+\sum_{j=1}^{n-3}\binom{n-j-1}{2}y^j.
\end{equation}
All empty sums are understood to be $0$.
In particular, when $n=3$, we have $T_{U_{3,3}}(x,y)=x^3$.
\end{lemma}

\begin{proof}
We begin by proving the formula \eqref{eq:tutte-rankthree}.
Since $\mathcal{M}$ has neither loops nor parallel elements,
every singleton and every two-element subset is independent.
Together with the independence of the empty set, this yields
$\operatorname{rk}_{\mathcal{M}}(A)=|A|$
for every subset $A\subseteq[n]$ with $|A|\leq 2$.
On the other hand, the definition of the uniform matroid gives
$\operatorname{rk}_{U_{3,n}}(A)=\min\{|A|,3\}$.
Thus, if $|A|\leq 2$, then
$\operatorname{rk}_{\mathcal{M}}(A)
=\operatorname{rk}_{U_{3,n}}(A)=|A|$.

For $|A|\geq 3$, the set $A$ contains a two-element independent set,
so $\operatorname{rk}_{\mathcal{M}}(A)\geq 2$.
As $\mathcal{M}$ has rank $3$, it is also true that
$\operatorname{rk}_{\mathcal{M}}(A)\leq 3$.
Hence $\operatorname{rk}_{\mathcal{M}}(A)\in\{2,3\}$,
whereas $\operatorname{rk}_{U_{3,n}}(A)=3$.

Since the two matroids share the ground set $[n]$ and both have rank $3$,
we may subtract the corresponding terms indexed by the same subset
$A\subseteq[n]$ in the formula \eqref{eq:tutte-subset-expansion}.
By the rank comparison above, the corresponding terms are equal
and therefore cancel whenever $|A|\leq 2$ or
$\operatorname{rk}_{\mathcal{M}}(A)=3$.
The subsets that remain are exactly those satisfying $|A|\geq 3$
and $\operatorname{rk}_{\mathcal{M}}(A)=2$.
Each of these subsets has rank $3$ in $U_{3,n}$.
Thus,
\begin{equation}\label{eq:deffi}
\begin{aligned}
T_{\mathcal{M}}(x,y)-T_{U_{3,n}}(x,y)
&=\sum_{\substack{A\subseteq[n]\\|A|\geq 3,\ \operatorname{rk}_{\mathcal{M}}(A)=2}}
\left((x-1)(y-1)^{|A|-2}-(y-1)^{|A|-3}\right)\\
&=(xy-x-y)\sum_{\substack{A\subseteq[n]\\|A|\geq 3,\ \operatorname{rk}_{\mathcal{M}}(A)=2}}(y-1)^{|A|-3}.
\end{aligned}
\end{equation}

Let $A\subseteq[n]$ satisfy $|A|\geq 3$ and
$\operatorname{rk}_{\mathcal{M}}(A)=2$.
By the basic properties of closure, $\overline{A}$ is a flat and
$\operatorname{rk}_{\mathcal{M}}(\overline{A})
=\operatorname{rk}_{\mathcal{M}}(A)=2$.
Hence $\overline{A}\in L_2(\mathcal{M})$.
Conversely, let $X\in L_2(\mathcal{M})$ and let $A\subseteq X$
with $|A|\geq 3$.
Since $\mathcal{M}$ is simple, $A$ contains a two-element independent
set, so $\operatorname{rk}_{\mathcal{M}}(A)\geq 2$.
Because $A\subseteq X$, we also have
$\operatorname{rk}_{\mathcal{M}}(A)
\leq\operatorname{rk}_{\mathcal{M}}(X)=2$.
Therefore, $\operatorname{rk}_{\mathcal{M}}(A)=2$.
As $X$ is a flat containing $A$, the minimality of the closure gives
$\overline{A}\subseteq X$.
Moreover, closure preserves rank, so
$\operatorname{rk}_{\mathcal{M}}(\overline{A})
=\operatorname{rk}_{\mathcal{M}}(X)=2$.
If $\overline{A}\subsetneq X$, then, since both sets are flats,
we would have
$\operatorname{rk}_{\mathcal{M}}(\overline{A})
<\operatorname{rk}_{\mathcal{M}}(X)$,
contradicting the preceding formula.
Hence, $\overline{A}=X$.

It follows that the subsets of rank $2$ and cardinality at least $3$
can be partitioned into disjoint classes according to their closures.
For each $X\in L_2(\mathcal{M})$ and each integer $j$ with
$3\leq j\leq |X|$, the inclusion $A\subseteq\overline{A}$ and
the preceding argument show that the $j$-element subsets with closure
$X$ are precisely all $j$-element subsets of $X$.
There are $\binom{|X|}{j}$ such subsets.
Since $\mathcal{M}$ has rank $3$, every $X\in L_2(\mathcal{M})$
is a proper subset of $[n]$, and hence $|X|\leq n-1$.
Thus, grouping the terms in the formula \eqref{eq:deffi} first according
to their closures and then according to the cardinalities of
these flats, we obtain
\begin{equation*}
\begin{aligned}
T_{\mathcal{M}}(x,y)-T_{U_{3,n}}(x,y)&=(xy-x-y)\sum_{\substack{X\in L_2(\mathcal{M})\\|X|\geq 3}}\sum_{j=3}^{|X|}\binom{|X|}{j}(y-1)^{j-3}\\
&=(xy-x-y)\sum_{p=3}^{n-1}c_p(\mathcal{M})\sum_{j=3}^{p}\binom{p}{j}(y-1)^{j-3}.
\end{aligned}
\end{equation*}
Rearranging proves the formula \eqref{eq:tutte-rankthree}.

We now prove the formula \eqref{eq:uniform-three}.
For each integer $m\geq 3$, let
\begin{equation*}
Q_m(y)
=\sum_{j=3}^{m}\binom{m}{j}(y-1)^{j-3}.
\end{equation*}
Using the rank function of $U_{3,n}$ in
the formula \eqref{eq:tutte-subset-expansion} and grouping subsets according
to their cardinalities, we obtain
\begin{equation*}
T_{U_{3,n}}(x,y)
=(x-1)^3+n(x-1)^2+\binom{n}{2}(x-1)+Q_n(y).
\end{equation*}
Using the recurrence relation for binomial coefficients,
we obtain, for every integer $m\geq 3$,
\begin{equation*}
\begin{aligned}
Q_{m+1}(y)
&=Q_m(y)
  +\sum_{j=3}^{m+1}\binom{m}{j-1}(y-1)^{j-3}\\
&=Q_m(y)+\binom{m}{2}+(y-1)Q_m(y)\\
&=yQ_m(y)+\binom{m}{2}.
\end{aligned}
\end{equation*}
Since $Q_3(y)=1$, induction gives
\begin{equation*}
\begin{aligned}
Q_n(y)&=\sum_{j=0}^{n-3}\binom{n-j-1}{2}y^j\\
&=\binom{n-1}{2}+\sum_{j=1}^{n-3}\binom{n-j-1}{2}y^j.
\end{aligned}
\end{equation*}
Substitute this identity into the expression for $T_{U_{3,n}}(x,y)$.
Since
\begin{equation*}
\begin{aligned}
(x-1)^3+n(x-1)^2+\binom{n}{2}(x-1)+\binom{n-1}{2}=x^3+(n-3)x^2+\binom{n-2}{2}x,
\end{aligned}
\end{equation*}
we obtain
\begin{equation*}
T_{U_{3,n}}(x,y)
=x^3+(n-3)x^2+\binom{n-2}{2}x
+\sum_{j=1}^{n-3}\binom{n-j-1}{2}y^j.
\end{equation*}
This proves the formula \eqref{eq:uniform-three}.
\end{proof}

\section{Proof of Theorem~\ref{thm:main}}

We now prove the main theorem.

\begin{proof}[\rm\textbf{Proof of Theorem \ref{thm:main}}]
Let $\mathcal{N}$ be an arbitrary simple matroid such that
$\mathrm{OS}(\mathcal{M})$ and $\mathrm{OS}(\mathcal{N})$
are isomorphic as ungraded $\mathbb{C}$-algebras.
It suffices to prove that
$T_{\mathcal{M}}(x,y)=T_{\mathcal{N}}(x,y)$.

Let $[n]$ and $[m]$ be the ground sets of $\mathcal{M}$ and
$\mathcal{N}$, respectively.
By the definition of the Orlik--Solomon algebra, both algebras
are connected graded $\mathbb{C}$-algebras, each generated by
finitely many homogeneous elements of degree $1$.
Hence, by Lemma~\ref{lem:graded}, there exists a graded
$\mathbb{C}$-algebra isomorphism
$\varphi\colon\mathrm{OS}(\mathcal{M})\to\mathrm{OS}(\mathcal{N})$.

Since $\varphi$ is a graded $\mathbb{C}$-algebra isomorphism,
its restriction $\varphi|_{\mathrm{OS}_p(\mathcal{M})}$
is a linear isomorphism from $\mathrm{OS}_p(\mathcal{M})$
to $\mathrm{OS}_p(\mathcal{N})$ for every integer $p\geq 0$.
Thus, Corollary~\ref{lem:rank} implies that $n=m$ and that
$\mathcal{M}$ and $\mathcal{N}$ have the same rank.
Let $r$ denote their common rank.
Since $\mathcal{M}$ has rank at most $3$, we have $0\leq r\leq 3$.
We consider the following three cases.

\textbf{Case 1.} $r=0$.

Since both matroids have rank $0$, any element of either ground set
would form a dependent singleton and would therefore be a loop,
contrary to simplicity.
Thus, $n=m=0$, and both matroids are the empty matroid.
By the subset expansion of the Tutte polynomial
in the formula \eqref{eq:tutte-subset-expansion}, we obtain
$T_{\mathcal{M}}(x,y)=T_{\mathcal{N}}(x,y)=1$.

\textbf{Case 2.} $1\leq r\leq 2$.

Since both matroids are simple, every singleton and every
two-element subset is independent in each of them.
Together with the independence of the empty set and the inequalities
$1\leq r\leq 2$, this implies that every subset of cardinality
at most $r$ is independent in both matroids.
On the other hand, by the definition of rank, every subset
of cardinality greater than $r$ is dependent.
Thus, the independent sets of both matroids are precisely
the subsets of $[n]$ of cardinality at most $r$.
Consequently, both matroids are the uniform matroid $U_{r,n}$,
and hence
$T_{\mathcal{M}}(x,y)=T_{U_{r,n}}(x,y)=T_{\mathcal{N}}(x,y)$.

\textbf{Case 3.} $r=3$.

By Lemma~\ref{lem:recovery} and the graded isomorphism $\varphi$,
we have $c_p(\mathcal{M})=c_p(\mathcal{N})$
for every integer $p\geq 2$.
Both matroids are simple matroids of rank $3$ on $n$ elements.
Applying Lemma~\ref{lem:Tutte} to both matroids and
subtracting the resulting formulas, we obtain
\begin{equation*}
T_{\mathcal{M}}(x,y)-T_{\mathcal{N}}(x,y)=(xy-x-y)\sum_{p=3}^{n-1}\bigl(c_p(\mathcal{M})-c_p(\mathcal{N})\bigr)\sum_{j=3}^{p}\binom{p}{j}(y-1)^{j-3}=0.
\end{equation*}
Therefore, $T_{\mathcal{M}}(x,y)=T_{\mathcal{N}}(x,y)$.

This completes the proof.
\end{proof}

Theorem~\ref{thm:main} gives the following two corollaries.

\begin{corollary}\label{cor:Falk}
Let $\mathcal{M}$ be a simple matroid of rank three. Then
$\mathrm{OS}(\mathcal{M})$ determines $T_{\mathcal{M}}(x,y)$.
\end{corollary}

\begin{corollary}(\cite{CordovilForge2003})\label{cor:real-realizable}
Let $\mathcal{M}$ be a simple matroid of rank at most three that is realizable
over $\mathbb{R}$. Then $\mathrm{OS}(\mathcal{M})$ determines
$T_{\mathcal{M}}(x,y)$.
\end{corollary}

\section{Conclusion}
We have established that the Tutte polynomial of every simple matroid with rank at most three
is determined by its Orlik--Solomon algebra.
This settles a conjecture of Falk~\cite{Falk1997} and extends the result of
Cordovil and Forge from simple matroids of rank at most three that are
realizable over $\mathbb{R}$ to all simple matroids of rank at most three.
We have also proved that the Orlik--Solomon algebra of a finite simple
rank-three matroid determines the number of rank-two flats of each
cardinality.
This pinpoints combinatorial information preserved by the algebra and points to
further investigation of the combinatorial invariants of higher-rank matroids that
can be recovered from their Orlik--Solomon algebras.


\begin{thebibliography}{99}



\bibitem{BellZhang2017}
J. Bell, J.J. Zhang, An isomorphism lemma for graded rings, Proc. Amer. Math. Soc. 145 (2017) 989--994.

%\bibitem{BergetEurSpinkTseng2023}
%A. Berget, C. Eur, H. Spink, D. Tseng, Tautological classes of matroids, Invent. Math. 233 (2023) 951--1039.

%\bibitem{BernardiKalmanPostnikov2022}
%O. Bernardi, T. K\'alm\'an, A. Postnikov, Universal Tutte polynomial, Adv. Math. 402 (2022) 108355.

%\bibitem{BrandenMoci2014}
%P. Br\"and\'en, L. Moci, The multivariate arithmetic Tutte polynomial, Trans. Amer. Math. Soc. 366 (2014) 5523--5540.

\bibitem{BrylawskiOxley1992}
T. Brylawski, J. Oxley, The Tutte polynomial and its applications, in: N. White (Ed.), Matroid Applications,
Encyclopedia of Mathematics and its Applications, Vol. 40, Cambridge University Press, Cambridge, 1992, pp. 123--225.

%\bibitem{ColbournProvanVertigan1995}
%C.J. Colbourn, J.S. Provan, D. Vertigan, The complexity of computing the Tutte polynomial on transversal matroids, Combinatorica 15 (1995) 1--10.

\bibitem{CordovilForge2003}
R. Cordovil, D. Forge, A note on Tutte polynomials and Orlik--Solomon algebras, European J. Combin. 24 (2003) 1081--1087.

\bibitem{EschenbrennerFalk1999}
C.J. Eschenbrenner, M.J. Falk, Orlik--Solomon algebras and Tutte polynomials, J. Algebraic Combin. 10 (1999) 189--199.

%\bibitem{Falk1993}
%M. Falk, Homotopy types of line arrangements, Invent. Math. 111 (1993) 139--150.

\bibitem{Falk1997}
M. Falk, Arrangements and cohomology, Ann. Comb. 1 (1997) 135--157.

\bibitem{Falk2001}
M. Falk, Combinatorial and algebraic structure in Orlik--Solomon algebras, European J. Combin. 22 (2001) 687--698.

\bibitem{FalkYuzvinsky2007}
M. Falk, S. Yuzvinsky, Multinets, resonance varieties, and pencils of plane curves, Compos. Math. 143 (2007) 1069--1088.

%\bibitem{KookReinerStanton1999}
%W. Kook, V. Reiner, D. Stanton, A convolution formula for the Tutte polynomial, J. Combin. Theory Ser. B 76 (1999) 297--300.

\bibitem{LibgoberYuzvinsky2000}
A. Libgober, S. Yuzvinsky, Cohomology of the Orlik--Solomon algebras and local systems, Compos. Math. 121 (2000) 337--361.

\bibitem{OrlikSolomon1980}
P. Orlik, L. Solomon, Combinatorics and topology of complements of hyperplanes, Invent. Math. 56 (1980) 167--189.


\bibitem{Oxley2011}
J. Oxley, Matroid Theory, second ed., Oxford University Press, Oxford, 2011.

%\bibitem{SheltonYuzvinsky1997}
%B. Shelton, S. Yuzvinsky, Koszul algebras from graphs and hyperplane arrangements, J. London Math. Soc. (2) 56 (1997) 477--490.

\bibitem{Yuzvinsky2001}
S. Yuzvinsky, Orlik--Solomon algebras in algebra and topology, Russian Math. Surveys 56 (2) (2001) 293--364.

\end{thebibliography}
\end{document}